\documentclass[11pt]{article}

\date{}  

\usepackage{pdfsync}
\usepackage{url}
\usepackage{cite}
\usepackage{plain}

\usepackage{wrapfig}
\usepackage{graphicx} 
\usepackage[english]{babel}
\usepackage[leqno]{amsmath}
\usepackage{amssymb}
\usepackage{verbatim}
\usepackage[mathscr]{eucal}

\usepackage{amstext}
\usepackage{makeidx}
\usepackage{amsthm}

\usepackage{hyperref}
\hypersetup{colorlinks,%
citecolor=red,%
linkcolor=blue,%
urlcolor=black}

\makeindex

\newtheorem{theorem}{Theorem} 

\newcommand{\mail}[1]{\href{unina:#1}{\texttt{#1}}}

\author{ Fabio De Angelis\thanks{Dept. of  Structures for Engineering and Architecture, University of Naples Federico II,   \newline
 Via Claudio 21, 80125, Naples, Italy.} \hspace{3mm}
    Monica De Angelis \thanks{ Dept. of Mathematics and Applications "R. Caccioppoli", University of Naples Federico II,  \newline
   Via Cinthia 26,80126, Naples, Italy.
\newline\mail{modeange@unina.it},}
}

\title{ On solutions related to FitzHugh-Rinzel type  model  }

\begin{document}

\maketitle






\maketitle

\begin{abstract}
A  ternary autonomous dynamical  system of  FitzHugh-Rinzel type is analyzed. The system, at start,  is reduced to a nonlinear integro differential equation. The fundamental solution $ H(x,t)  $ is explicitly determined and the initial value problem is analyzed  in the whole  space. The solution is expressed by means of  an  integral equation  involving  $ H(x,t)  $. Moreover, adding an extra  control term,  explicit solutions  are   achieved. 
\end{abstract}

\section{Introduction }
\label{intro}

The FitzHugh-Rinzel  (FHR) system   \cite{bbks,ws15,ZB,k2019} is  a three-dimensional model deriving  from the FitzHugh-Nagumo  (FHN) model \cite{i,dr13,krs,nono,r1,r2,mda13,glrs} 
developed  to incorporate bursting phenomena of nerve cells. Indeed,  a number of different  cell types exhibit a behaviour characterized by brief bursts of oscillatory activity alternated  by quiescient periods during which the membrane potential only  changes slowly, and this behaviour is called bursting, see e.g. \cite{ks}. Accordingly, bursting oscillations are characterized  by a variable of the system  that  changes periodically from an active phase of rapid spike oscillations to a silent phase.  These phenomena are  becoming increasingly important as they are being investigated  in many scientific fields. 
   Indeed, phenomena of  bursting  have been observed as  electrical behaviours in many nerve and
endocrine cells such as hippocampal and  thalamic neurons, mammalian midbrain  and pancreatic in  $ \beta -$ cells, see e.g. \cite{bbks} and references therein. Also, in the cardiovascular system, bursting oscillations
are generated by the electrical activity of cardiac cells that  excite the heart membrane
to produce the contraction of ventricles and auricles \cite{qmv}. Furthermore, bursting oscillations represent a topic of  potential  interest in   dynamics and bifurcation mechanisms of  devices and  structures and in   
the analysis of nonlinear problems in mechanics \cite{dcgda18,sandra,dea18,candd19,sw}. Recent studies proved that the development of this field helps in the  studying of  the restoration  of  synaptic  connections. Indeed, it seems that nanoscale memristor devices  have potential to reproduce the behaviour of a biological synapse \cite{jnbbikem,clag}.   This would lead in the  future,  also in case of traumatic lesions, to the introduction of electronic synapses to connect neurons directly.

The paper is organized as follows. In section 1.1 the mathematical problem is defined and the state of the art with  the aim of the paper are discussed. In section 2,  the explicit expression  of the  fundamental solution $ H(x,t)  $ is achieved and  some   properties  are proved. In section 3 the integral solution for    the initial value problem is given. In section 4 the insertion of  an extra term  allows us to obtain  explicit solutions for the model.

\subsection{Mathematical considerations, state of the art and aim of the paper}
Generally, denoting by $ D, \varepsilon, \beta, c $ constant parameters, the  (FHN) model  is a p.d.e. system   such that
\begin{equation}
\label{FHN}
  \left \{
   \begin{array}{lll}
    \displaystyle{\frac{\partial \,u }{\partial \,t }} =\,  D \,\frac{\partial^2 \,u }{\partial \,x^2 }
     \,-\, w\,\, + f(u ) \,  \\
\\
\displaystyle{\frac{\partial \,w }{\partial \,t } }\, = \, \varepsilon (-\beta w +c +u),
  \end{array}
  \right.
\end{equation}                                                     
where function $\, f(u)\,$ depends on the reaction kinetics of the model. In  the literature   $ f(u) $ can  assume  a piecewise linear form, see, e.g. \cite{drw}   and reference therein,  or  $ f(u) = u-u^3/3$ \cite{glrs}.    
However, in general,  one has \cite{i,ks}:
\begin{equation}                 \label{12}
f(u)= u\, (\, a-u \,) \, (\,u-1\,) \,\,\,\,\,\,\, \,\, (\,0\,<\,a\,< \,1\,).
\end{equation}

As for the FitzHugh-Rinzel  model,  most of the  articles consider the following system characterized by three o.d.e.:

\begin{equation}
\label{FHR}
  \left \{
   \begin{array}{lll}
    \displaystyle{\frac{\partial \,u }{\partial \,t }} =\, u-u^3/3 + I_{ext} 
     \,-\, w\,\,+y\,\, \,  \\
\\
\displaystyle{\frac{\partial \,w }{\partial \,t } }\, = \, \varepsilon (-\beta w +c +u)
\\
\\
\displaystyle{\frac{\partial \,y }{\partial \,t } }\, = \,\delta (-u +h -dy)
   \end{array}
  \right.
\end{equation}                                                      
where    $ I_{ext}, \varepsilon ,\beta,c,d,  h,\beta ,\delta  $  indicate arbitrary constants.

In this paper, in order to evaluate the contribute of a diffusion term,  the following  FitzHugh-Rinzel  type system  is   considered:

\begin{equation}
\label{11}
  \left \{
   \begin{array}{lll}
    \displaystyle{\frac{\partial \,u }{\partial \,t }} =\,  D \,\frac{\partial^2 \,u }{\partial \,x^2 }
     \,-\, w\,\,+y\,\, + f(u ) \,  \\
\\
\displaystyle{\frac{\partial \,w }{\partial \,t } }\, = \, \varepsilon (-\beta w +c +u)
\\
\\
\displaystyle{\frac{\partial \,y }{\partial \,t } }\, = \,\delta (-u +h -dy).
   \end{array}
  \right.
\end{equation}

Indeed,  the second order term with    $ D > 0 $  represents just the diffusion contribution and it can be associated to the axial current in the axon. It derives from the  Hodgkin-Huxley (HH) theory  for nerve membranes where, if  $ d $ represents the axon diameter and $ r_i$ is the resistivity, the spatial variation in the potential $ V  $ gives the term $ (d/4 r_{i}) V_{xx} $  from  which term $ D\, u_{xx} $  derives \cite{m2}.
  
Moreover it is also assumed   $ \,\beta, \,\,d,\,  \varepsilon, \,\,\delta \,\, $  as  positive constants that together  with  $  c,\,\, h,\, $ characterize the model's kinetic.

Model (\ref{11}) can be considered as  a two time-scale slow-fast system  with  two fast variables $(u,w)$ and one slow variable $(y)$. However, if  for instance, $ \varepsilon = \delta $ the system can be considered as a two time-scale with one fast variable $u $ and two slow variables $(w,y)$. Otherwise, if $ \delta $ and $ \varepsilon  $ have significant difference, it can also be considered as a three-time-scale system with the fast variable $ u, $ the intermediate variable and the slow
 variable \cite{xxcj}.

As  for  function $f (u)$ one  considers  the non-linear form expressed in formula (\ref{12}). As a consequence, it results

\begin{equation}      \label{14}
f(u)\, =\,-\,a\,u\, +\, \varphi(u) \quad with \quad  \varphi(u) \,=\, u^2\, (\,a+1\,-u\,)\quad  0<a<1 
\end{equation}

 Then, the system (\ref{11}) becomes

\begin{equation}
\label{15}
  \left \{
   \begin{array}{lll}
    \displaystyle{\frac{\partial \,u }{\partial \,t }} =\,  D \,\frac{\partial^2 \,u }{\partial \,x^2 } -au \,\,
     \,-\, w\,\,+y\,\, + \varphi(u ) \,  \\
\\
\displaystyle{\frac{\partial \,w }{\partial \,t } }\, = \, \varepsilon (-\beta w +c +u)
\\
\\
\displaystyle{\frac{\partial \,y }{\partial \,t } }\, = \,\delta (-u +h -dy).
   \end{array}
  \right.
\end{equation}

Indicating by means of

\begin{equation}      \label{ic}
u(x,0)\, =\,u_0 \,, \qquad w(x,0)\, =\,w_0  \qquad   y(x,0)= y_0,\qquad\qquad ( \,x\,  \in \Re\,)
\end{equation}
the initial values, from   $(\ref{15})_{2,3} $   it follows:

\begin{equation}
\label{17}
\left \{
   \begin{array}{lll}
\displaystyle  w\, =\,w_0 \, e^{\,-\,\varepsilon \beta\,t\,} \,+\, \frac{c}{\beta}\,( 1- e^{- \, \varepsilon \, \beta \,t} )\,+ \varepsilon \int_0^t\, e^{\,-\,\varepsilon \,\beta\,(\,t-\tau\,)}\,u(x,\tau) \, d\tau 
  \\
  \\
\displaystyle  y\, =\,y_0 \, e^{\,-\,\delta\, d\,t\,} \,+\, \frac{h}{d}\,( 1- e^{ - \, \delta \,d  \,t} )\,- \delta \int_0^t\, e^{\,-\,\delta \,d\,(\,t-\tau\,)}\,u(x,\tau) \, d\tau.
 \end{array}
  \right.  
\end{equation}

Consequently, denoting the source term by

\begin{equation}
\label{18}
\displaystyle F(x,t,u) =\varphi (u)  - w_0(x)  e^{- \varepsilon  \beta  t}+ y_0(x)  e^{- \delta  d  t}- \frac{c}{\beta}( 1- e^{- \varepsilon \beta t} )+\frac{h}{d}( 1- e^{ -\delta d t} ),
\end{equation}
problem (\ref{15})-(\ref{ic}) can be modified into  the  following  initial value problem   $\,{\cal P}$:

	  \begin{equation}                                                     \label{19}
  \left \{
   \begin{array}{lll}
   \displaystyle  u_t - D  u_{xx} + au + \int^t_0 [ \varepsilon e^{- \varepsilon  \beta (t-\tau)}+ \delta e^{- \delta d (t-\tau)} ]u(x,\tau)  d\tau = F(x,t,u) \\    
\\ 
 \displaystyle \,u (x,0)\, = u_0(x)\,  \,\,\,\,\, x\, \in \, \Re. 
   \end{array}
  \right.
 \end{equation}

As for the state of art,   mathematical considerations  allow   to  assert that the knowledge of the fundamental solution  $ H(x,t) $ related to the
linear parabolic operator $ { L }: $

 \begin{equation} \label{a1}
 {L }u\equiv  u_t - D  u_{xx} + au + \int^t_0 [ \varepsilon \,e^{-\, \varepsilon  \beta (t-\tau)}\,+ \delta e^{-\, \delta d (t-\tau)} ]u(x,\tau) \, d\tau , 
 \end{equation} 
  leads   to determine   the solution  of  $\,{\cal P}$. Indeed,  if $\, F(\,x,\,t,\,u\,)\,$  verifies appropriate assumptions, through the fixed point theorem,  solution can be expressed by means of   an  integral equation, see f.i. \cite{c,dr8}.

Moreover,  according to \cite{dr8},
when operator $ L  $ assumes a similar but simpler form,  many properties and inequalities are    achieved.

The aim of the  paper is to  explicitly determine   the fundamental solution $ H(x,t)  $ which   involves naturally the  diffusion constant $ D.  $ Then, the initial value problem in all  the space is analyzed and the solution is deduced by means of an  integral equation.  Moreover, using  a method of travelling wave,  solutions  of  a modified  FitzHugh-Rinzel  type system have  been  explicitly determined pointing out  the influence of the  diffusion parameter D.

\section{Fundamental solution  and its properties}

Indicating by $\,T\,$   an arbitrary  positive constant, let us  consider the  initial- value problem
   (\ref{19}) defined in the whole space $  \Omega_T: $
\begin{center}

$\  \Omega_T =\{(x,t) :  x \in \Re
, \  \ 0 < t \leq T \},$
\end{center}
 and let us denote  by

\[\hat u (x,s) \, = \int_ 0^\infty \, e^{-st} \, u(x,t) \,dt \,\,, \,\,\,  \,\hat F (x,s)   \, = \int_ 0^\infty \,\, e^{-st} \,\, F\,[x,t,u (x,t)\,] \,dt \,,\,\]
 the Laplace transform with respect to $\,t.\,$ If $\,\hat  H (x,s)\,$ expresses   the  $ {\cal L}_t$ transform    of  the fundamental solution  $\,  H ( x,t),\, $ from (\ref{19}) it follows:

 \begin{equation}                                                     \label{21}
   {\hat u }(x,s)\,=\,  \int_\Re \, \hat H \,(\, x-\xi, s\,)\, [\,u_0(\,\xi\,) \,+\,\hat F(\xi,s)\,]\,d\xi\,, 
\end{equation} 
 and  {\em  formally}   it follows that

\begin{eqnarray}  \label{A14}
 & \displaystyle \nonumber u (x,t)  =\int_\Re   \,H ( x-\xi, t)\,\, u_0 (\xi)\,\,d\xi \,+
 \\
 \\
 &\displaystyle \nonumber\,+\,\int ^t_0     d\tau \int_\Re   H ( x-\xi, t-\tau)\,\, F\,[\,\xi,\tau, u(\xi,\tau\,)\,]\,\, d\xi.
\end{eqnarray}

So that,  denoting by     $ J_1 (z) \,$  the  Bessel function of first kind and order $\, 1,\,$  let us  consider the following  functions:

   \begin{eqnarray} \label {22}
\nonumber \displaystyle  &H_1(x,t)\,=\, \, \, \frac{e^{- \frac{x^2}{4\,D\, t}\,}}{2 \sqrt{\pi  D t } }\,\,\, e^{-\,a\,t}\,+
 \, 
 \\
 \\
&\nonumber \displaystyle - \frac{1}{2} \,\,    \,\,\,\int^t_0  \frac{e^{- \frac{x^2}{4 \, D\, y}\,- a\,y}}{\sqrt{t-y}} \,\, \, \frac{\,\sqrt{\varepsilon} \,\, e^{-\beta \varepsilon \,(\, t \,-\,y\,)}}{\sqrt{\pi \, D \,}} J_1 (\,2 \,\sqrt{\,\varepsilon \,y\,(t-y)\,}\,\,)\,\,\} dy,
   \end{eqnarray}

\begin{equation} \label{H2}
 \displaystyle H_2 =\int _0 ^t  H_1(x,y) \,\,e^{ -\delta d (t-y)} \sqrt{\frac{\delta y}{t-y}}   J_1( \,2 \,\sqrt{\,\delta \,y\,(t-y)\,}\,\,\, dy.
 \end{equation} 

Besides, by setting
\begin{eqnarray}
  \displaystyle   \sigma^2 \ \,=\, s\, +\, a \, + \, \frac{\delta}{s+\delta d}\, + \, \frac{\varepsilon}{s+\,\beta \varepsilon},\,\,
\end{eqnarray}
and by denoting  

\begin{equation} \label{A16}
\displaystyle H = H_1  - H_2,
\end{equation}
the following theorem holds:

\begin{theorem}

  In the half-plane $ \Re e  \,s > \,max(\,-\,a ,\,-\beta\varepsilon ,- \delta d\,)\,$    the Laplace integral  $\,{\cal L  }_t\, H \, \,$  converges absolutely for all  $\,x>0,\,$   and it results:
 \end{theorem}

\begin{equation}      \label{24}
\displaystyle \,{\cal L  }_t\,\,H\,\equiv \,\, \hat H \, =\int_ 0^\infty e^{-st} \,\, H\,(x,t) \,\,dt \,\,= \frac{1}{\sqrt{D}} \, \frac{e^{- \frac{|x|}{\sqrt{D}}\,\sigma }}{2 \,\sigma \,  }.
\end{equation}

Moreover, function H(x,t)  satisfies some  properties typical of the  fundamental solution of heat equation, such as:

 a) \hspace{1mm} $ H( x,t) \, \, \in  C ^ {\infty}, \,\,\,\,$ \,$\,\,\, t>0, \,\,\,\, x \,\,\, \in \Re, $ 

b) \hspace{1mm} for fixed $\, t\,>\,0,\,\,\, H \,$  and its derivatives are vanishing esponentially fast as $\, |x| \, $  tends to infinity.

c) \hspace{1mm} In addition, it results   $ \displaystyle\lim _{t\, \rightarrow 0}\,\,H(x,t)\,=\,0, $ for any fixed $\, \eta \,>\, 0,\, $ uniformly for all $\, |x| \,\geq \, \eta.\, $

\begin{proof}

Since for all real $\, z\,$ one has $\, |J_1\,(z)|\, \leq \, 1 , $   the Fubini -Tonelli theorem assures that

\[ \,\, L_t \bigg( \frac{1}{2} \,\,    \,\,\,\int^t_0  \frac{e^{- \frac{x^2}{4 D y}\,- a\,y}}{\sqrt{t-y}} \,\, \, \frac{\,\sqrt{\varepsilon} \,\, e^{-\beta \varepsilon \,(\, t \,-\,y\,)}}{\sqrt{\pi \, D \,}} J_1 (\,2 \,\sqrt{\,\varepsilon \,y\,(t-y)\,}\,\,\bigg)\,dy \,= \]

\[ = \,\, {\frac{ \sqrt \varepsilon  }{2\, \sqrt{\pi  D  }} }\, \,\,\,\ \int ^\infty_0   e^{- (s+a)y \,-\frac{ x^2}{4 D y} } 
\,\,dy\,\,\int ^\infty_0 e^{-(s+\beta \varepsilon )t}\, \,\,J_1 (2 \sqrt{\varepsilon  \,y\,t  }\,) \,\,\frac{dt}{\sqrt t}\,\,  \]
 and being 

\begin {equation} \label{25}
 \,\, \int ^\infty_0    \,\, {e^{- p\, t} }\,\,  \sqrt{\frac{c}{ t}}\,\,\,\,J_1 ( \, 2 \sqrt{c\,t}\,\,)\,\,dt \,=\, \, 1\,-\, e^{-\, c/ p\,}\,\,\,\,\,\,\, (\,\Re e  \,p > \,0\,), 
 \end{equation}

 \begin{equation} \label{26}
\int _0^\infty  \frac{e^{\,-x^2/4t\, }}{\, \sqrt{ \pi \,t\,}}\,  e^{- (s+\alpha) t}\,\, dt = \frac{e^{\,-x\, \sqrt{\,s+\alpha\,}}}{\, \sqrt{ \, s+\alpha\,}}\,\,  
\end{equation} 
 it results:

\begin{equation} \label{27}
  \hat H_1 (x,s)= \frac{1}{2 \sqrt{\pi D }}\int_ 0^\infty e^{-\frac{x^2}{4Dy}- (s+a+\frac{ \varepsilon}{s+\beta\varepsilon})y}\frac{dy}{\sqrt y}=  \frac{1}{2  \sqrt D}\frac{e^{- \frac{|x|}{\sqrt{D}}\,r}}{r } 
\end{equation} 
where 

\begin{equation}
 r^2= s\, +\, a \, + \, \frac{\varepsilon}{s+\beta \varepsilon}.\,\,
\end{equation}

Besides,
since Fubini -Tonelli theorem and (\ref{25})   one has: 

\begin{equation}
\displaystyle \hat H_2 = \hat H_1 -\frac{1}{2\, \sqrt{\pi D }}\,\int_ 0^\infty \,\,e^{\,-\,\frac{\,x^2\,}{4 Dy}\,-\, (s+a+\frac{\, \varepsilon\,}{s+\beta\varepsilon\,}+\frac{\delta}{s+ d \delta }\,)\,y\,}\,\,\frac{dy\,}{\sqrt y\,}\,\,
\end{equation}
from which, since (\ref{26}),

\begin{equation} \label{29}
\,  \hat H \, (x,s)\,=\,  \frac{1}{2 \,\, \sqrt D \, \,}\,\,\,\frac{e^{- \frac{\,|x|\,}{\sqrt{D}}\sigma }}{\,\sigma  \,}
\end{equation}
is deduced.

Besides, by means of property of convolution  for which $ f*g=g*f, $   since  (\ref{22}) and (\ref{H2}), property a) is evident. Moreover, properties b) and c) are proved following theorem 3.2.1 of \cite{c}. In particular, as for property c), for $ |x|\geq \eta $ and since $ |J_1 (z) |\leq 1, $ it results:

\begin{equation}
\displaystyle  |H_1|\,=\, \, \, \frac{e^{- \frac{\eta^2}{4\,D\, t}\,}}{2 \sqrt{\pi  D t } }\,\,\, e^{-\,a\,t}\,+\sqrt{\frac{\varepsilon \,t}{\pi D}};
\end{equation}

\begin{equation} 
 \displaystyle| H_2 |\,\, \leq  \sqrt{\frac{\delta \, t}{4\pi D}}  +  2 \,  \sqrt{\frac{\varepsilon   \delta }{\pi D}} \, \,t
 \end{equation} 
 from which property  follows.
\hbox{}
 \hfill
\rule {1.85mm}{2.82mm}
\end{proof}

Now,  introducing the following functions:

\begin{equation}  \label{A22}
 \left \{
   \begin{array}{lll}
 \displaystyle \varphi(x,t) =  \frac{e^{- \frac{x^2}{4\,D\, t}\,}}{2 \sqrt{\pi  D t } }\,\,\, e^{-\,a\,t};
\\ 
\\
 \displaystyle\psi_\varepsilon (y,t) =\, \frac{\,\sqrt{\varepsilon \,y} \,\, e^{-\beta \varepsilon \,(\, t \,-\,y\,)}}{\sqrt{t-y \,}} J_1 (\,2 \,\sqrt{\,\varepsilon \,y\,(t-y)\,};\,\,
\\ 
\\
  \displaystyle\nonumber\psi_\delta (y,t) = \frac{\,\sqrt{\delta \,y} \,\, e^{-\delta \,d  \,(\, t \,-\,y\,)}}{\sqrt{t-y \,}} J_1 (\,2 \,\sqrt{\,\delta\,y\,(t-y)\,};\,\,
   \end{array}
  \right.
\end{equation}
it results:

\begin{equation} \label{211}
H_1(x,t)= \varphi(x,t)- \, \int _0^t \varphi(x,y) \,\, \psi_\varepsilon (y,t)\, dy
\end{equation}

\begin{equation} \label{212}
H_2(x,t)=   \int _0^t   H_1(x,y) \,\,\psi_\delta (y,t)\,\, dy.
\end{equation}
Moreover, by denoting

\begin{equation}
\,\, g_1(x,t)\,\ast g_2(x,t)  = \int _0^t g_1(x,t-\tau) g_2(x,\tau) \,d\tau
\end{equation} 
the convolution with respect to $ t, $ for $\,t\,>\,0,\,$ as  proved in  \cite{dr8} by means of formula (20),(21) and (24),  it results:

\begin{equation} \label{215}
(\partial_t +a -D\partial_{xx}) H_1 = -\varepsilon e^{-\varepsilon \beta t } \ast H_1(x,t) = -\varepsilon  K_{\varepsilon}(x,t)
\end{equation}  
 where $\, K_ {\varepsilon}\, $ is given by 

\begin{equation}
K_{\varepsilon} ( x,t) \, = \, \frac{1}{2 \, \sqrt{\pi \, D}} \,\, \int^t_0 \,\, e^{- \frac{x^2}{4 D y}\,\,-a \,y\,-\,\beta \varepsilon (t-y)}\ \, \, J_0 \,(2 \,\sqrt{\varepsilon\,y\,(t-y)}\,)\,\,\frac{dy}{\sqrt{y}}.
\end{equation}     \label {216}  
Hence,  the following theorem can be proved:

\begin{theorem}
 For $\,t\,>\,0,\,$ it results   $ {L }\,H =0, $ i.e.

\begin{equation} \label{a11}
  H_t - D  H_{xx} + aH + \int^t_0 [ \varepsilon \,e^{-\, \varepsilon  \beta (t-\tau)}\,+ \delta e^{-\, \delta d (t-\tau)} ]H(x,\tau) \, d\tau  =0. 
 \end{equation}
\end{theorem} 
\begin{proof}
Let us consider that:

\begin{eqnarray}
\nonumber &(\partial_t +a -D\partial_{xx}) H_2 =  H_1(x,t) \psi_{\delta}(t,t) +\int_0^t H_1(x,y) \,\big[\, \partial_t \psi_\delta (y,t)+a \psi_\delta (y,t)\big]   
\\
\\
\nonumber &- \, D\,\int _0^t \, \psi_\delta (y,t)\,\, \partial_{xx}H_1 (x,y)\,\, dy.
\end{eqnarray}
Accordingly, given relation (\ref{215}),  one has

\begin{eqnarray}
\nonumber &(\partial_t +a -D\partial_{xx}) H_2 =  \int _0^t \big[ H_1(x,y) \,\, \partial_t \psi_\delta (y,t) - \psi_\delta (y,t) \,\,\partial_y H_1(x,y) \, \big] dy + 
\\  \nonumber\\
&+H_1(x,t) \psi_{\delta}(t,t) - \displaystyle \varepsilon \int_0^t  K_ {\varepsilon} (x,y) \psi_\delta (y,t) dy .
\end{eqnarray}
Besides, considering that: 

\begin{equation}
\int _0^t \psi_\delta (y,t) \partial_y H_1(x,y) \,  dy =  H_1(x,t) \psi_{\delta}(t,t) -\int _0^t H_1(x,y) \,\, \partial_y \psi_\delta (y,t) \,  dy, 
\end{equation}
 one has: 

\begin{eqnarray}
\nonumber  \displaystyle &(\partial_t +a -D\partial_{xx}) H_2 =  \int _0^t  H_1(x,y) \,\,\big[ \partial_t \psi_\delta (y,t) + \partial_y\psi_\delta (y,t)  \big] dy  
\\  \nonumber\\
&- \displaystyle \varepsilon \int_0^t  K_ {\varepsilon} (x,y) \psi_\delta (y,t) dy   
\end{eqnarray}
where it results:

\begin{equation}
 \partial_t \psi_\delta (y,t) + \partial_y\psi_\delta (y,t) =\delta  e^{-\delta d  \,(\, t \,-\,y\,)}J_0 (\,2 \,\sqrt{\,\delta \,y\,(t-y)\,}.\, 
 \end{equation}
So that, denoting by

\begin{eqnarray} \label{38}
 K_ \delta (x,t) \equiv \int^t_0 \, e^{- \delta d \,(t- y)}\,  H_1 (x,y)  \, J_0 \, (\, 2\, \sqrt{\delta\, y ( t-y)  }\,) \,\,dy, 
\end{eqnarray}
one has:

\begin{eqnarray}
\nonumber  \displaystyle &(\partial_t +a -D\partial_{xx}) H_2 =   \delta K_\delta -  \varepsilon \int_0^t  K_ {\varepsilon} (x,y) \psi_\delta (y,t) dy.
\end{eqnarray}

Consequently, since for equation   (\ref{215}) one has $K_ {\varepsilon} (x,t) =    \, H_{1}(x,t)\ast e^{\,-\varepsilon \beta \,t}, $ 
by means of Fubini -Tonelli theorem and (\ref{212}), it is  proved that: 

\begin{eqnarray} \label{aaa}
  \displaystyle &(\partial_t +a -D\partial_{xx}) H_2 =   \delta K_\delta - \varepsilon \,e^{-\varepsilon \beta t } \ast H_2(x,t).\,
\end{eqnarray}

On the other hand, the convolution $ e^{-\,\delta d  \,t} * H(x,t) $ is given by 

\[e^{-\delta d t} * H(x,t) = e^{-\delta d t} *H_1(x,t) - \int^t_0  H_1 (r,y) dy  \int^t_y  e^{-\delta d  (t-\tau)} \psi_\delta (y,\tau)\, d\tau \]  with

\begin{equation}                              \label{37}
  \int^t_y    e^{- \delta d (t-\tau)}  \psi_\delta  ( y,\tau) d\tau=   e^{- \delta d (t-y)}  \int ^t_y \sqrt{\frac{{\delta \,y}}{\tau -y } }\,\,\,J_1 \, (\, 2\, \sqrt{\delta\, y ( \tau-y)  }\,)\,\, d\tau =
\end{equation}
\[ \,=\,\, e^{\,-\, \delta d \,(t-y)} \,  \ 
\Bigl[
\,1\,-\ \,J_0 \, (\, 2\, \sqrt{\delta \, y ( t-y)  }\,)\,\,
\Bigr]. \]
As  a consequence, it results:

\begin{equation}                              \label{317}
  e^{-\,\delta d \,t} \ast\,H =   K_\delta.
\end{equation}

Therefore, given relations (\ref{215}),  (\ref{aaa}), (\ref{317}),  theorem holds.  
\hbox{}
 \hfill\rule {1.85mm}{2.82mm}
 
\end{proof}

\section{Solution related to the (FHR) problem}

To  provide  the solution by means of the integral expression (\ref{A14}), some convolutions need to be  determined.

In order to  evaluate  $  \int_0^t\,d\tau\, \int_\Re \,H(\xi,\,\tau) \, \,d\xi , $  let us start to observe that

\begin{eqnarray}
 \nonumber &\displaystyle \int_\Re  d\xi   \int_0^t   H_2 (\xi,\tau)d\tau = \int_\Re  d\xi  \int_0^t H_1 (\xi,y) dy \int_y^t \psi_\delta (y,\tau) d\tau
\end{eqnarray}
with 

\begin{eqnarray} 
&\nonumber \displaystyle\int_y^t \psi_\delta (y,\tau) d\tau=  1- e^{-\delta \,d(t-y) } J_0 ( 2\sqrt{\delta y(t-y)} + 
\\
\\
&\nonumber \displaystyle -\delta \, d\int _y^t  e^{-\delta \,d(t-y) } J_0 ( 2\sqrt{\delta y(t-y)}d\tau. 
\end{eqnarray}

Consequently, for   (\ref{38}), one has:

\begin{eqnarray}
 \nonumber & \displaystyle\int_\Re  d\xi   \int_0^t  H_2 (\xi,\tau) d\tau  = \int_\Re d\xi \int _0^t   H_1 (\xi,\tau) d\tau -  \int_\Re  K_\delta  (\xi,t) d\xi \,+
\\
\\
&\nonumber \displaystyle\,- \,\delta d  \,\int_\Re   d\xi   \int_0^t     K_\delta (\xi, \tau)   d\tau
\end{eqnarray}
So that, according to (\ref{A16}), it results: 

\begin{equation} \label{A50}
\displaystyle\int_0^t\,d\tau\, \int_\Re \,H(\xi,\,\tau) \, \,d\xi =  \int_\Re  K_\delta  (\xi,t) d\xi \,+\delta d \, \int_\Re  d\xi \int_0^t   K_\delta (\xi, \tau) d\tau  .
\end{equation}

Now, let us evaluate

\begin{eqnarray}
 &\nonumber \displaystyle  e^{-\,\varepsilon\,\beta \,t} *\,H_2 =  \int^t_0 \,H_1 \,(x,y)\, dy \, \int^t_y \, e^{-\,\beta \varepsilon  \,(t-\tau)\,}\psi_\delta \,(y,\tau)\, d\tau. 
 \end{eqnarray}
Considering  $(\ref{A22})_3$, after an integration by parts, one obtains:

\begin{eqnarray}
 &\nonumber \displaystyle  e^{-\varepsilon\beta t} *H_2 =  \displaystyle  e^{-\beta \varepsilon  t} *H_1 (x,t) \,-  \int^t_0 e^{-\delta d (t-y)}  H_1 (x,y) J_0  ( 2 \sqrt{\delta y ( t-y)  }) dy+
\nonumber\\
\\\nonumber
 &\nonumber \displaystyle  
  (\varepsilon \beta - \delta d)  \int^t_0   d\tau \int^\tau_0   H_1 \,(x,y)\, e^{-\delta d (\tau-y)}   \, e^{- \varepsilon \beta \,(t-\tau)}\, J_0 \, (\, 2\, \sqrt{\delta\, y ( \tau-y)  }\,) dy,  
\end{eqnarray}

and, for (\ref{38}), it results:

\begin{eqnarray} \label{A40}
e^{-\,\varepsilon\,\beta \,t} *\,H_2=  e^{-\,\varepsilon\,\beta \,t} *\,H_1- K_\delta+ (\varepsilon \beta - \delta d) e^{-\beta \varepsilon  \,t} * K_\delta.
\end{eqnarray}

Moreover, since (\ref{A16}) and (\ref{A40}), one deduces:

\begin{eqnarray} \label{A41}
e^{-\,\varepsilon\,\beta \,t} *\,H=   K_\delta+ (\delta d-\varepsilon \beta ) e^{-\beta \varepsilon  \,t} * K_\delta.
\end{eqnarray}

Now, let  us denote  by

\begin{equation}
\,\, f_1(x,t)\,\star f_2(x,t)  = \int_\Re \, f_1(\xi, t) f_2(x-\xi,t) \,d\xi
\end{equation}

 the convolution with respect to the space $ x, $ 
and let

\begin{equation}
H  \otimes F\,=\, \int_0^t\,d\tau\, \int_\Re \,H(x-\xi,t-\tau) \, \,F \,[\,\xi, \tau ,u(\xi,\tau)\,]\,d\xi.
\end{equation}

Considering  (\ref{317}) and (\ref{A41}) one has:

\begin{equation}
\label{A51}
\left \{
   \begin{array}{lll}
\displaystyle H  \otimes \, e^{-\delta d t }=\int_\Re   K_\delta (\xi,t)\,\,d\xi,
\\
\\ 
 \displaystyle H  \otimes \, e^{- \beta \varepsilon\, t } =\int_\Re \big[  K_\delta+ (\delta d-\varepsilon \beta ) e^{-\beta \varepsilon  \,t} * K_\delta\,\big]  d\xi.
  \end{array}
  \right.  
  \end{equation}
  
  Moreover, it results:

\begin{equation}
\label{315}
\left \{
   \begin{array}{lll}
H  \otimes  ( y_0(x)\, e^{-\delta d t })= y_0 \star  K_\delta
  \\
  \\
H  \otimes  ( w_0(x) \,e^{-\beta \varepsilon t })= w_0 \star [  K_\delta+( \delta d-\varepsilon \beta ) e^{-\beta \varepsilon  \,t} * K_\delta\,]  
 \end{array}
  \right.  
\end{equation}

where

\begin{equation}
\label{316}
  w_0 \star  ( \delta d-\varepsilon \beta ) e^{-\beta \varepsilon  \,t} * K_\delta\, =  ( \delta d-\varepsilon \beta ) \,  w_0 \, \otimes ( e^{-\beta \varepsilon  \,t}  K_\delta\,).
\end{equation}

Consequently, given  (\ref{18}) and  (\ref{A14}), for (\ref{A50}),(\ref{A51}),(\ref{315}) and (\ref{316}), it results:

\begin{eqnarray}  \label{318}
 \nonumber & \displaystyle u(x,t) \,=\, u_0 (x) \star H \,  +  (y_0(x) -\, w_0(x)) \star  K_\delta + \, \,\varphi (u) \otimes  H +\,
\\ \nonumber
\\ 
&\nonumber\displaystyle  +(\varepsilon \beta - \delta d) w_0(x) \otimes e^{-\beta \varepsilon  \,t} K_\delta+  \frac{c}{\beta}\big ( \delta d - \varepsilon \beta )e^{-\beta \varepsilon  \,t} \otimes  K_{\delta}  + 
\\
\\
\nonumber &\displaystyle  + \bigg (\frac{h}{d}\,- \frac{c}{\beta}\,\bigg)  \delta  d \,\otimes \, K_\delta  
\end{eqnarray}
and   this formula, together with  relations (\ref{17}), allow us  to determine  also  $\, v(x,t) \,$  and $\, y(x,t) \,$  in terms of the data.

\section{Explicit solutions}

Several methods have been developed to find exact solutions  related to  partial differential equations \cite{dma18,Lg,mda19,adc,k18}. In this case, by referring to  \cite{krs},   an extra term is added in order to achieve some    solutions.
Accordingly, let us consider

\begin{equation}
\label{A60}
  \left \{
   \begin{array}{lll}
    \displaystyle{\frac{\partial \,u }{\partial \,t }} =\,  D \,\frac{\partial^2 \,u }{\partial \,x^2 }
     \,-\, w\,\,+y\,\, + f(u ) \,  \\
\\
\displaystyle{\frac{\partial \,w }{\partial \,t } }\, = \, \varepsilon (-\beta w +c +u)+ k\, u^2
\\
\\
\displaystyle{\frac{\partial \,y }{\partial \,t } }\, = \,\delta (-u +h -dy)
   \end{array}
  \right.
\end{equation}                                                     
where $ k\neq0 $ and let us assume  $ \varepsilon \beta = \delta \, d $  and   $ f(u) =  2\,u\,(a-u)\,(u-1).  $

Under these conditions, problem (\ref{19})  turns into: 

 \begin{equation}                                                     \label{A61}
  \left \{
   \begin{array}{lll}
   \displaystyle  u_t - D  u_{xx} + 2au + (\varepsilon + \delta) \int^t_0  \,e^{-\, \varepsilon  \beta (t-\tau)}\,u(x,\tau) \, d\tau \,=\, F(x,t,u) \,\\    
\\ 
 \displaystyle \,u (x,0)\, = u_0(x)\,  \,\,\,\,\, x\, \in \, \Re, 
   \end{array}
  \right.
 \end{equation}
 where, by denoting   $\displaystyle  \varphi_1 = 2\,  u^2\, (\,a+1\,-u\,) + k \int_0^t\, \,e^{\,-\,\varepsilon \,\beta\,(\,t-\tau\,)}\,  \, u^2(x,\tau)\,\,  d\tau, \, $ it results:

\begin{equation}
\label{188}
\displaystyle F =\varphi_1 (u)  - w_0 e^{- \varepsilon  \beta  t}+ y_0  e^{- \delta  d  t}- \frac{c}{\beta}( 1- e^{- \varepsilon \beta t} )+\frac{h}{d}( 1- e^{ -\delta d t} )
\end{equation}
and

\begin{equation}
\label{AA17}
\left \{
   \begin{array}{lll}
\displaystyle  w =w_0  e^{-\varepsilon \beta t} + \frac{c}{\beta}( 1- e^{-  \varepsilon  \beta t} )\,+ \int_0^t e^{-\varepsilon \beta(t-\tau)} \big[\,\varepsilon\, u(x,\tau) + k  u^2(x,t) \big ] d\tau 
  \\
  \\
\displaystyle  y\, =\,y_0 \, e^{\,-\,\delta\, d\,t\,} \,+\, \frac{h}{d}\,( 1- e^{ - \, \delta \,d  \,t} )\,- \delta \int_0^t\, e^{\,-\,\delta \,d\,(\,t-\tau\,)}\,u(x,\tau) \, d\tau. 
 \end{array}
  \right.  
\end{equation}

In order to find explicit solutions, let us  introduce 

\[  z=x- C \,t,\]
obtaining,  from  system (\ref{A60}), the following equation:

\begin{eqnarray} \label{41}
\nonumber & \displaystyle  D\,C\, u_{zzz} +(C^2-\varepsilon \beta\,D ) u_{zz} - 6 C  u^2 u_{z} +4  C (a+1) u \, u_{z} +2 \varepsilon \beta u^3  +  k\,u^2
\\
\\
\nonumber &  \displaystyle - \, C\, (2\,a\,+ \varepsilon \beta)\, u_z  - 2\,\varepsilon \beta (a+1) u^2 + 2\,  \varepsilon \beta \,a \,u  + (\varepsilon +\delta)\,u + \varepsilon \,c - \delta h=0.
\end{eqnarray}
Now, let us  consider

\begin{equation} \label{A65}
f(z)= \sqrt{y} \,\,\tanh \,(\sqrt{ y}\, (z-z_0)),
\end{equation}
that is a solution of Riccati type  equation:

\[ f_{z} + f^2 -y =0, \]
and let us assume 
   
\begin{equation} \label{53}
 u(z) = A\,\, f(z)+\,b.
\end{equation}

Since
 \begin{eqnarray} \label{55}
 \nonumber & u_z = \, - A \,f ^2(z) +A\,y  
 \\
 \nonumber & u_{zz} =2\,A f^3-2\,A\,f\,y
 \\
  \nonumber & u_{zzz} = - 6 A f^4(z) +8 \, A\, y f^2(z) - 2\, A \, y^2
   \\
   & u \, u_z = - A^2\, f^3(z) - A b f^2  + A^2 y\,f + A by 
  \\
  \nonumber &u^2  \,u_z  = - A^3 f^4 - 2 A^2\,b\, f^3 +( A^3 y - A b^2 ) f^2 + 2 A ^2 b\,y f +A b^2 y, 
  \end{eqnarray}
in order  to satisfy equation (\ref{41}),  
constant $ b  $ needs to assume the following  expression:

\begin{equation}\label{56}
b=  \frac{\varepsilon \beta \,A}{2C}-\frac{ \varepsilon\, +\, \delta }{2\,k} 
\end{equation}
and, moreover,  it has to  be 
\begin{eqnarray}
 \nonumber & A^2 = D
 \\ 
 &\nonumber \displaystyle  a+1= 3\, b + \frac{ C\, A \,  }{ 2 \,D}
 \\ 
 &\nonumber \displaystyle D\, y= {3 } \,b^2 + \bigg(\frac{C}{A} -3- \frac{k}{\varepsilon \beta }\bigg)\, b -\frac{C }{2 A}- \frac{\varepsilon +\delta }{2\, \varepsilon\, \beta} + 1
  \end{eqnarray}
and 
\begin{eqnarray}
\nonumber \displaystyle & \varepsilon \,c-\delta\,h = 2\, C DA y^2  + \varepsilon \beta \, C A y - 2 \varepsilon \beta\,b^3 + 2\,\varepsilon \beta \, (a+1) \, b^2 +
\\
&\nonumber \displaystyle- 2\, \varepsilon \beta\, a\,b+ 6 \, C\,A\,y\,b^2- 4 C\,y \,(a+1) A b+ 2\,a\, C\,A\,y\,- ( \varepsilon + \delta) \,b - k\,b^2.
\end{eqnarray}

\begin{figure}
\centering
{\includegraphics  [width=80mm] 
{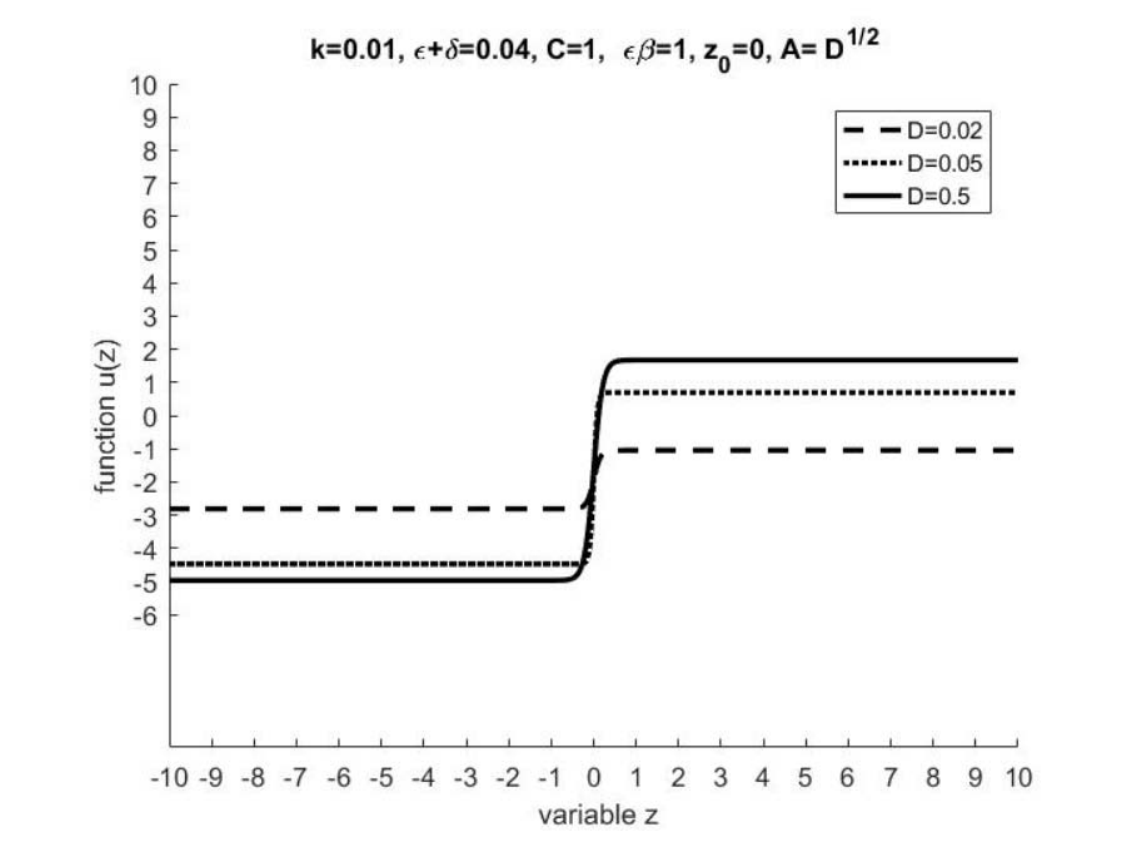}} 
\caption {Solution u(z) when  $ \varepsilon \beta=1, $  $z_0=0  $, $ \varepsilon + \delta = 0.04, k= 0.01  ,C=1, $  $  A = \sqrt{D}  $ and  for  $D=0.02,  D=0.05,D=0.5. $  }
\label{figure 1}
\end{figure}

So, if for instance, it is assumed    $ \varepsilon \beta=1, $  $z_0=0  $, $ \varepsilon + \delta = 0.04, k= 0.01 $ and $ C=1 $  with $ D> 0.019, $ for  $ A = \sqrt{D}, $ by introducing

\begin{equation}
g(D)= \sqrt{3900\sqrt{D} -1501 \,D + 150 \,D \, \sqrt{D}-500},
\end{equation}
equation (\ref{53}) gives

\begin{equation} \label{63}
u(z)= \frac {\sqrt{2} }{20\, D^{1/4}} \, \,g(D)  \, \tanh\bigg(\frac{\sqrt{2}\,\,z\,g(D)}  { 20 \, D^{3/4}\,} \, \bigg) + \frac{\sqrt{D}}{2} -2. 
\end{equation}
 
 In Fig.1, solutions u(z) expressed by means of  formula (\ref{63}) are illustrated for different values of  $D, $ 
 by  showing  that  the amplitude increases as $ 0<D <1 $ increases.

\textbf{Remark} When 
 fast variable $ u$ simulates the membrane potential of a nerve cell, while  slow
variable $ w $ and superslow   variable $ y $ determine the corresponding ion number densities,  model (\ref{A60}) with its solutions  can be of interest in applications  to understand how impulses are propagated from one neuron to another.  Moreover,  as  similarly underlined in \cite{k2019}, the  knowledge of  exact solutions   can help  in testing different applications of models in neuroscience.

\textbf{Acknowledgements}

The present work has been  developed with the economic support of MIUR (Italian Ministry of University and Research) performing the  activities of the project   
ARS$01_{-}  00861$  “Integrated collaborative 
 systems for smart factory - ICOSAF".

\hspace{-6mm} The  paper has been performed under the auspices of G.N.F.M. of INdAM.

\hspace{-6mm} The authors are grateful to anonymous referees for their helpful comments.


\begin {thebibliography}{99} 




 
 \bibitem{bbks}R. Bertram, T. Manish J. Butte,T. Kiemel
and A. Sherman,
 Topological and phenomenological classification of bursting oscillations,  Bull. Math.Biol, Vol. 57, No. 3, pp. 413, (1995) 

 
 \bibitem{ws15} J.Wojcik, A. Shilnikov, Voltage Interval Mappings for an Elliptic
Bursting Model in Nonlinear Dynamics New Directions
Theoretical Aspects
 González-Aguilar H; Ugalde E. (Eds.)  219 pp  Springer,(2015)   





\bibitem{ZB}   Zemlyanukhin A. I., Bochkarev A. V.,  Analytical Properties and Solutions of the FitzHugh – Rinzel Model, Rus. J. Nonlin. Dyn.,	 Vol. 15, no. 1, pp.  3-12, (2019)


\bibitem{k2019}N. A. Kudryashov, On Integrability of the FitzHugh – Rinzel Model,Russian Journal of Nonlinear Dynamics, 2019, vol. 15, no. 1, pp. 13–19.


\bibitem {i}Izhikevich E.M., Dynamical Systems in Neuroscience: The Geometry of Excitability and Bursting,p.397. The MIT press. England (2007)




















\bibitem{dr13}M.De Angelis, P. Renno, { Asymptotic effects of boundary perturbations in excitable systems,} Discrete and continuous dynamical systems series B, 19, no 7 (2014)


\bibitem{krs}N.K.Kudryashov, K.R. Rybka, A. Sboev,
 Analytical properties of the perturbed FitzHugh-Nagumo model,
Applied Mathematics Letters, 76 , 142-147 ,(2018)





\bibitem{nono} M. De Angelis, On a model of superconductivity and biology, Advances and Applications in Mathematical Sciences
Volume 7, Issue 1,  Pages 41-50, (2010)









\bibitem{r1}Rionero, S.  A rigorous reduction of the $L\sp 2$-stability of the solutions to a nonlinear binary reaction-diffusion system of PDE's to the stability of the solutions to a linear binary system of ODE's, J. Math. Anal. Appl. { 319 } no. 2, 377-397 (2006)


\bibitem{r2}Rionero, S, Torcicollo,I. On the dynamics of a nonlinear reaction–diffusion duopoly mode,International Journal of Non-Linear Mechanics,Volume 99,  Pages 105-111,(2018)




\bibitem{mda13}M. De Angelis, Asymptotic Estimates Related to an Integro Differential Equation, Nonlinear Dynamics and Systems Theory, 13 (3)  217–228 (2013)

   
   
   \bibitem{glrs}G. Gambino,M. C. Lombardo,G. Rubino, M. Sammartino, Pattern selection in the 2D FitzHugh–Nagumo model, Ric. di Mat. https://doi.org/10.1007/s11587 1–15 (2018)

\bibitem{ks}  Keener, J. P. Sneyd,J. {  Mathematical Physiology }. Springer-Verlag, N.Y, 470 pp, (1998)
   
    
 
\bibitem{qmv} A. Quarteroni, A. Manzoni and C. Vergara,
The cardiovascular system:Mathematical modelling, numerical algorithms and clinical applications,
 Acta Numerica, vol 26,pp. 365-590,(2017)





\bibitem{dcgda18} F. De Angelis, D. Cancellara, L. Grassia, A. D'Amore,  The influence of loading rates on hardening effects in elasto/viscoplastic strain-hardening materials,   Mechanics of Time-Dependent Materials,   22 (4)  533-551, (2018)  


    
\bibitem{sandra}Carillo, S.; Chipot, M.; Valente, V.; Vergara Caffarelli, G. On weak regularity requirements of the relaxation modulus in viscoelasticity. Commun. Appl. Ind. Math. 10 , no. 1, 78–87,(2019)



\bibitem{dea18}  F. De Angelis, D. Cancellara, Dynamic analysis and vulnerability reduction of asymmetric structures: Fixed base vs base isolated system, Composite Structures, 219 203-220, (2019).   

\bibitem{candd19}  F. De Angelis, Extended formulations of evolutive laws and constitutive
relations in non-smooth plasticity and viscoplasticity,
Composite Structures, 193 35-41, (2018). 













 
\bibitem{sw} H. Simo, P. Woafo,Bursting oscillations in electromechanical systems, Mechanics Research Communications 38  537 – 541, (2011)



\bibitem{jnbbikem}E.Juzekaeva,  A. Nasretdinov,  S. Battistoni, T. Berzina,  S. Iannotta,  R. Khazipov,  V. Erokhin,  M. Mukhtarov,  Coupling Cortical Neurons through Electronic Memristive Synapse,  Adv. Mater. Technol.  4, 1800350 (6) (2019)

\bibitem{clag} 
F. Corinto, V. Lanza, A. Ascoli, and Marco Gilli,  Synchronization in Networks of FitzHugh-Nagumo
Neurons with Memristor Synapses, in  20th European Conference on Circuit Theory and Design (ECCTD)  IEEE. (2011) 


\bibitem {drw} M De Angelis, A priori estimates for excitable models,Meccanica, Volume 48, Issue 10, pp 2491–2496 (2013) 
 
\bibitem {m2} Murray, J.D.   {   Mathematical Biology I,  }. Springer-Verlag, N.Y, 767 pp, (2003)



\bibitem{xxcj} Wenxian Xie, Jianwen Xu, Li Cai, Yanfei Jin,  Dynamics and Geometric Desingularization of the Multiple Time Scale FitzHugh Nagumo Rinzel Model with Fold Singularity , Communications in Nonlinear Science and Numerical Simulation, 63, p. 322-338.  (2018)





\bibitem{c}J. R. Cannon, The one-dimensional heat equation, Addison-Wesley Publishing Company  483 pp, (1984)





\bibitem  {dr8}De Angelis, M. Renno, P, { Existence, uniqueness and a priori estimates for a non linear integro - differential equation, }Ricerche di Mat. 57  95-109 (2008)

\bibitem{dma18} M. De Angelis,  { On the transition from parabolicity to hyperbolicity
for a nonlinear equation under Neumann boundary
conditions,} Meccanica, Volume 53, Issue 15, pp 3651–3659, (2018)

\bibitem{Lg}Li H., Guoa Y.:  New exact solutions to the Fitzhugh Nagumo equation, Applied Mathematics and Computation 
{\bf180},  2, 524-528 (2006) 


\bibitem{mda19} M. De Angelis,  A wave equation perturbed by viscous terms: fast and
slow times diffusion effects in a Neumann problem,Ricerche di Matematica,  Volume 68, Issue 1, pp 237–252, (2019)



\bibitem{adc} B.Prinaria, F.Demontis, Sitai Li, T.P.Horikis, Inverse scattering transform and soliton solutions for square matrix nonlinear Schrödinger equations with non-zero boundary conditions,Physica D: Nonlinear Phenomena,Volume 368,   Pages 22-49, (2018)

\bibitem{k18}N.K.Kudryashov,  Asymptotic and Exact Solutions of the FitzHugh–Nagumo Model, Regul. Chaotic Dyn., vol 23,	No 2, 152–160, (2018)










\end{thebibliography}
\end{document}